\documentclass{amsart}

\usepackage{enumerate}
\usepackage{amssymb}
\usepackage{amsmath,amscd}
\usepackage{amsthm}
\usepackage{graphicx}
\usepackage{pictexwd,dcpic}
 \usepackage{verbatim}

\newcommand{\R}{\mathbb{R}}
\newcommand{\C}{\mathbb{C}}
\newcommand{\HH}{\mathbb{H}}
\newcommand{\Z}{\mathbb{Z}}
\newcommand{\K}{\mathbb{K}}
\newcommand{\OO}{\operatorname{O}}
\newcommand{\SO}{\operatorname{SO}}
\newcommand{\GL}{\operatorname{GL}}
\newcommand{\U}{\operatorname{U}}
\newcommand{\SU}{\operatorname{SU}}
\newcommand{\Sp}{\operatorname{Sp}}
\newcommand{\F}{\mathcal{F}}
\newcommand{\srf}{singular Riemannian foliation}

\newcommand{\Hom}{\operatorname{Hom}}
\newcommand{\End}{\operatorname{End}}
\newcommand{\Sym}{\operatorname{Sym}}
\newcommand{\Cl}{\operatorname{Cl}}

\newcommand{\Le}{\mathcal{L}}

\newcommand{\Av}{\operatorname{Av}}
\newcommand{\diam}{\operatorname{diam}}
\newcommand{\Span}{\operatorname{span}}
\newcommand{\Spin}{\operatorname{Spin}}
\newcommand{\Pin}{\operatorname{Pin}}
\newcommand{\Hess}{\operatorname{Hess}}
\newcommand{\Gr}{\operatorname{Gr}}

\newtheorem{theorem}{Theorem}

\newtheorem{corollary}[theorem]{Corollary}
\newtheorem{lemma}[theorem]{Lemma}
\newtheorem{proposition}[theorem]{Proposition}

\newtheorem{maintheorem}{Theorem}

\theoremstyle{definition}
\newtheorem{definition}[theorem]{Definition}

\theoremstyle{remark}
\newtheorem{remark}[theorem]{Remark}
\newtheorem{example}[theorem]{Example}

\title[Foliations and quadratic polynomials]{Singular Riemannian foliations and their quadratic basic polynomials}

\author[R.~Mendes]{Ricardo Mendes$^*$}
\address{Mathematisches Institut, WWU M\"unster, Germany}
\email{rmendes@gmail.com}

\author[M.~Radeschi]{Marco Radeschi$^*$}
\address{Mathematisches Institut, WWU M\"unster, Germany}
\email{mrade\_02@uni-muenster.de}

\thanks{$^*$ received support from SFB 878: \emph{Groups, Geometry \& Actions}}


\begin{document}

\maketitle
\begin{abstract}
We present a new link between the Invariant Theory of infinitesimal \srf s and Jordan algebras. This, together with an inhomogeneous version of Weyl's First Fundamental Theorems, provides a characterization of the recently discovered Clifford foliations in terms of basic polynomials. This link also yields new structural results about infinitesimal foliations, such as the existence of non-trivial symmetries.
\end{abstract}


\section{Introduction}

Singular Riemannian foliations are certain partitions $\F$ of Riemannian manifolds $M$ into \emph{leaves}, which are smooth, connected, locally equidistant submanifolds of $M$ (see \cite{Molino} for the precise definition). When all the leaves have the same dimension, the foliation is called \emph{regular}, and these have been a classical object of study in Riemannian Geometry since the 1950's (see for example  \cite{Haefliger58, Reinhart59}). Moreover, \srf s may also be considered  as generalizations of  other classical objects, such as Isoparametric Foliations (see for example \cite{Thorbergsson00}); and the orbit decompositions of $M$ by isometric actions of connected Lie groups $G$ (these foliations are called \emph{homogeneous}).

In the present article we consider \emph{infinitesimal} \srf s with closed leaves. That is, closed \srf s $\F$ of a Euclidean space $V$, where the origin is a leaf. For simplicity, we will refer to these as \emph{infinitesimal foliations}. These correspond bijectively to closed \srf s of round spheres, via the Homothetic Transformation Lemma \cite{Molino}. Infinitesimal foliations generalize the orbit decomposition of orthogonal representations of compact, connected Lie groups $G$. Moreover, as in the homogeneous case, there is a Slice Theorem that describes small tubular neighbourhoods of leaves in a closed \srf\ $(M,\F)$, up to foliated diffeomorphism, from  data including a certain infinitesimal foliation called the \emph{slice foliation}, see \cite{MendesRadeschi15}. 

We may define, for any infinitesimal foliation $\F$ (or, more generally, for any partition) of $V$, the algebra $\R[V]^\F$ of \emph{basic polynomials}, that is, polynomials that are constant along the leaves of $\F$. A further similarity between orthogonal representations of compact groups and infinitesimal foliations is that $\R[V]^\F$ is finitely generated, and the leaves are common level sets of the generating polynomials (this is known as Hilbert's Theorem in the homogeneous case). We will refer to this fact as \emph{Algebraicity}, see \cite{LytchakRadeschi15}, and Theorem \ref{T:algebraicity} below. Note also that such a generating set determines the algebra of \emph{smooth} basic functions on $V$, and hence, via the Slice Theorem, the smooth basic functions on a tubular neighbourhood of a closed leaf in any (non-infinitesimal) \srf\ (see \cite{MendesRadeschi15}).

More generally, any set $\mathcal{P}\subset\R[V]$ of polynomials defines a partition of $V$ into its common level sets, which we denote $\Le(\mathcal{P})$. In this language, the Algebraicity Theorem above is equivalent to $\F=\Le(\R[V]^\F)$, which naturally leads to the question: what subsets $\mathcal{P}$ define infinitesimal foliations? Our first main result states that, for any given $(V,\F)$, the homogeneous basic polynomials of degree two, denoted $\R[V]^\F_2$, form such a set, up to connected components. 
\begin{maintheorem}
\label{MT:srf}
Let $(V,\F)$ be an infinitesimal foliation without trivial factors. Then the connected components of the leaves of  $\F_2=\Le(\R[V]^\F_2)$ form an infinitesimal foliation of $V$. 
\end{maintheorem}

In fact we present a classification of the foliations $\F_2$, which is then used to prove Theorem \ref{MT:srf}. The building blocks are \emph{Clifford foliations} (introduced in \cite{Radeschi14}, see Example \ref{E:Clifford}), given by $\Le(|x|^2,\left<P_0x,x\right>,\ldots \left<P_mx,x\right>)$, where $P_0,\ldots P_m\in \Sym^2(\R^{2l})$ form a Clifford system; and the standard diagonal representations of $\OO(k)$, $\U(k)$, $\Sp(k)$ on $n$ copies of $\R^k$, $\C^k$, $\HH^k$ respectively.

\begin{maintheorem}[Classification]
\label{MT:classification}
Let $(V,\F)$ be an infinitesimal foliation without trivial factors. Then $(V,\F_2)$ is isomorphic to a product of Clifford foliations and (orbit decompositions of) standard diagonal representations.
\end{maintheorem}

When $\F$ is homogeneous, $\F_2$ is again homogeneous and can be computed using Schur's lemma. It decomposes as a product with factors corresponding to isotypical components, and each factor is one of the standard diagonal representations above depending on the number and type of irreducible components (see Proposition \ref{P:homogeneous}). 

%

The proof of Theorem \ref{MT:classification} starts from the observation that $\R[V]^\F_2$ has the structure of a \emph{Jordan algebra}. Then one invokes the classification of special, formally real Jordan algebras \cite{JvNW34, Albert34}. 

As a particular case of Theorem \ref{MT:classification}, if $(V,\F)$ is an infinitesimal foliation such that $\R[V]^\F$ is generated by quadratic polynomials, then $\F=\F_2$ must be one of the foliations listed there. The next theorem shows that the converse holds as well.

\begin{maintheorem}[First Fundamental Theorem]
\label{MT:fft}
Let $(V,\F)$ be a product of Clifford foliations and (orbit decompositions of) standard diagonal representations. Then $\R[V]^\F$ is generated in degree two.
\end{maintheorem}

In the special case of standard diagonal representations, the complex version of this result was proved by Weyl \cite{Weyl} under the name of First Fundamental Theorem of $\OO(n,\C)$, $\GL(n,\C)$, $\Sp(n,\C)$ respectively. The real version then follows easily. In a similar way, to prove the result for Clifford foliations we introduce the notion of \emph{complexification} of an infinitesimal foliation.

Theorems \ref{MT:classification} and \ref{MT:fft} give a simple characterization of Clifford foliations, as those inhomogeneous infinitesimal foliations whose algebra of basic polynomials is generated in degree two. The only exceptions come in three families, which coincide with the standard actions of $\OO(k)$, $\U(k)$, $\Sp(k)$ on \emph{two} copies of $\R^k$, $\C^k$, $\HH^k$.

Among the foliations $\F_2$ listed in Theorem \ref{MT:classification}, the ones with disconnected leaves are precisely the ones given by $\OO(k)$ acting on $(\R^k)^n$ for $n\geq k$, or the Clifford foliations $C_{3,1}$ and $C_{7,1}$ (see Example \ref{E:Clifford}). Replacing the leaves of these foliations with their connected components, Theorem \ref{MT:fft} may not hold anymore: for example replacing the $\OO(3)$-action on $(\R^3)^3=\R^{3\times 3}$ with the $\operatorname{SO}(3)$ sub-action, a new degree-three invariant appears, namely the determinant.

To describe the geometric significance of $\F_2$, we define a subspace $W\subset V$ to be invariant if it is the union of leaves of $\F$ (that is, a saturated set). Then $\R[V]^\F_2$ and $\F_2$ ``encode'' the structure of the set of all $\F$-invariant subspaces, and we arrive at the following:
\begin{maintheorem}[Geometric characterization]
\label{MT:characterization}
Let $(V,\F)$ be an infinitesimal foliation. Then the foliation by the connected components of $\F_2$ is the coarsest infinitesimal foliation that has the same invariant subspaces as $\F$.
\end{maintheorem}

In addition to the existence of invariant subspaces, non-triviality of $\F_2$ imposes additional structure on the geometry of the original foliation $\F$, which generalizes the notions of isotypical components and type. More precisely, the isotypical components are defined as the indecomposable factors of $\F_2$, and their type as real, complex, quaternionic or Clifford according to the classification in Theorem \ref{MT:classification}. Moreover, non-triviality of $\F_2$ implies the existence of enough symmetries of $(V,\F)$ to act transitively on the connected components of the moduli space of invariant subspaces: 
\begin{maintheorem}[Symmetry]
\label{MT:symmetry}
Let $(V,\F)$ be an infinitesimal foliation, and $U,W\subset V$ be invariant subspaces in the same isotypical component and with the same dimension. Then there exists a foliated linear isometry $g:V\to V$ such that $g(U)=W$.
\end{maintheorem}

An important open question in the area of \srf s is that of \emph{Smoothness of Isometries}, see \cite{AlexandrinoLytchak11}, which may be phrased as follows. Does the metric structure on the leaf space $V/\F$ determine its smooth structure, or, equivalently via \cite{MendesRadeschi15}, the algebra $\R[V]^\F$? We obtain the partial answer that $\R[V]^\F_2$ is determined by the metric structure of the leaf space $V/\F$:
\begin{maintheorem}
\label{MT:smoothness}
Let $(V,\F)$ and $(V',\F')$ be infinitesimal foliations without trivial factors, and $\phi:V/\F\to V'/\F'$ an isometry. Then $\phi^*$ induces an  isomorphism of Jordan algebras $\R[V']^{\F'}_2\to\R[V]^\F_2$.
\end{maintheorem}

The present article is organized as follows. Section \ref{S:basic} contains general remarks about the algebra of basic polynomials, including a description of degree-one invariants. Then in Section \ref{S:classification} we make the link with Jordan algebras, and use it prove Theorems \ref{MT:srf} and \ref{MT:classification}. Section \ref{S:invariant} contains definitions and some properties of invariant subspaces and isotypical components for infinitesimal foliations, including the proof of Theorem \ref{MT:characterization}. In Section \ref{S:symmetries} we produce some symmetries of an infinitesimal foliation, leading to the proof of Theorem \ref{MT:symmetry}. Section \ref{S:leafspace} is devoted to proving Theorem \ref{MT:smoothness}. In Section \ref{S:fft} we recall Weyl's First Fundamental Theorems and extend them to the case of Clifford foliations, giving a proof of Theorem \ref{MT:fft}. Finally, Appendix \ref{A:jordan} contains the basic facts about Jordan algebras that are used in this article, especially in Section \ref{S:classification}, while Appendix \ref{A:symmetry} describes the set of all foliated linear maps between possibly different infinitesimal foliations as an algebraic variety.


\section{Algebras of basic polynomials}
\label{S:basic}

Let $(V,\F)$ be an \emph{infinitesimal foliation}, that is, a singular Riemannian foliation with closed leaves on the Euclidean space $V$, such that $0$ is a leaf. Denote by $\R[V]^\F$ the algebra of \emph{$\F$-basic} (or simply \emph{basic}) polynomials, that is, polynomials on $V$ which are constant on the leaves of $\F$. By the Homothetic Transformation Lemma (see \cite{Molino}, Lemma 6.2), $\R[V]^\F$ is a graded subalgebra of $\R[V]$, so that $\R[V]^\F=\oplus_{i=0}^\infty \R[V]^\F_i$ is the direct sum of the spaces $\R[V]^\F_i$ of degree $i$ homogeneous basic polynomials. Moreover $\R[V]^\F$ determines $\F$ in the sense that it separates leaves:

\begin{theorem}[Algebraicity, \cite{LytchakRadeschi15}]
\label{T:algebraicity}
Let $(V,\F)$ be an infinitesimal foliation. Then the algebra $\R[V]^\F$ of basic polynomials is finitely generated, and if $\rho_1, \ldots, \rho_N$ is a set of generators, then every leaf of $\F$ is of the form $\rho^{-1}(y_1, \ldots y_N)$ for some $(y_1,\ldots y_N)\in\R^N$, where $\rho=(\rho_1, \ldots , \rho_N):V\to \R^N$.
\end{theorem}

\begin{remark}
\label{R:averaging}
The main tool in the proof of Theorem \ref{T:algebraicity} is the averaging operator $\Av:C^0(V)\to C^0(V)^\F$ defined by
\begin{equation*}
\Av(f)(x)=\frac{1}{\operatorname{vol}(L_x)}\int_{y\in L_x}f(y)dy
\end{equation*}
where $dy$ denotes the Riemannian volume form on the leaf $L_x$, and $\operatorname{vol}(L_x)$ the volume of $L_x$ with respect to $dy$. The key property of $\Av$ is that it preserves the space of smooth functions, and hence the space of homogeneous polynomials of degree $d$, see \cite{LytchakRadeschi15}. As in the homogeneous case, the usefulness of the averaging operator goes beyond the proof of Theorem \ref{T:algebraicity}. Indeed, it is used several times in this article. \end{remark}

\begin{remark}
\label{R:integrallyclosed}
In addition to averaging, another way of generating basic polynomials is to use the fact that $\R[V]^\F$ is integrally closed in $\R[V]$. That is, if $f\in\R[V]$ satisfies a monic polynomial equation $f^n+a_{n-1}f^{n-1}+\cdots+ a_0=0$ with coefficients $a_i\in\R[V]^\F$, then $f\in\R[V]^\F$. Indeed, the coefficients are constant on each leaf, and the leaves are connected, hence $f$ is also constant on each leaf.

As an illustration, consider the actions of $\OO(2)$ and $\SO(2)$ on the space of $2\times 2$ matrices
$\begin{pmatrix}
x & z \\
y & w
\end{pmatrix}$
by left multiplication. The polynomials $x^2+y^2,z^2+w^2,xz+yw$ are $\OO(2)$-invariant, hence $\SO(2)$-invariant (see Example \ref{E:standard})
. On the other hand, the determinant $xw-yz$ is only $\SO(2)$-invariant. It satisfies the monic equation $(xw-yz)^2=(x^2+y^2)(z^2+w^2)-(xz+yw)^2$. 
\end{remark}

\begin{definition}
\label{D:levelsets}
Let $V$ be a vector space and $\mathcal{P}\subset\R[V]$ an arbitrary set of polynomials on $V$. Define $\Le(\mathcal{P})$ to be the partition of $V$ into the  common level sets of $\mathcal{P}$, that is, the partition given by the equivalence relation 
\[ x\simeq y\ \Longleftrightarrow \ f(x)=f(y) \ \forall f\in\mathcal{P}.\]
\end{definition}

The Algebraicity Theorem implies that every infinitesimal foliation is of the form $\Le(\mathcal{P})$ for some finite $\mathcal{P}$. If $\F$ is the orbit decomposition of a $G$-action, a set $\mathcal{P}$ such that $\Le(\mathcal{P})=\F$ is called a ``separating set'', see \cite{Kemper09}. Note that the equivalence relation above is the zero set in $V\times V$ of  $\delta(\mathcal{P})$, where $\delta(f(x))=f(x)-f(y)$. This means that, by the real Nullstellensatz (see \cite[Theorem 4.1.4]{BCR}), $\Le(\mathcal{P})=\Le(\mathcal{P}')$ if and only if the ideals in $\R[V\times V]$ generated by $\delta(\mathcal{P})$ and $\delta(\mathcal{P}')$ have the same real radical, where  $\mathcal{P}, \mathcal{P}'\subset\R[V]$.

Since $\R[V]^\F$ determines $\F$, and is a graded algebra, it makes sense to explore the geometric meaning of each graded part $\R[V]^\F_i$. The degree one part corresponds to the trivial factors of $\F$, a result that is well-known in the homogeneous case (see \cite[page 75]{GorodskiLytchak14}):
\begin{proposition}[Trivial factors]
\label{P:trivial}
Let $(V,\F)$ be an infinitesimal foliation and let $SV$ denote the unit sphere in $V$. Then the following are equivalent:
\begin{enumerate}[a)]
\item $(V,\F)$ splits off a trivial factor.
\item $\R[V]^\F_1\neq 0$.
\item $\diam(SV/\F)=\pi$.
\item $\diam(SV/\F)>\pi/2$.
\end{enumerate}
\end{proposition}
\begin{proof}
(a)$\Rightarrow$(b) If $(V,\F)=(V_0\oplus V_1,\F_0\times\text{trivial})$, then any linear functional whose kernel contains $V_0$ is basic.

(b)$\Rightarrow$(a) Let $\lambda\in\R[V]^\F_1$ be non-zero. We may assume $\lambda(x)=\left< x, u\right>$ for some unit vector $u\in V$. Then $V_0=\ker(\lambda)=u^\perp$ is a union of leaves, hence $(V_0,\F|_{V_0})$ is an infinitesimal foliation, and the same applies to the other level sets $\lambda^{-1}(y)$, for $y\in\R$. We show that the restriction of the foliation $\F$ to the level sets are translations of $(V_0,\F|_{V_0})$. Indeed, if $L\subset V_0$ is a leaf, then $yu+L$ is the intersection of $\lambda^{-1}(y)$ with the set of all points at distance $|y|$ from $L$, and therefore $yu+L$ is a union of leaves. By reversing the roles of $0$ and $y$, it follows that $yu+L$ is actually a single leaf. Therefore $(V,\F)=(V_0,\F|_{V_0})\times(\R,\text{trivial})$.

(b)$\Rightarrow$(c) Let $\lambda:V\to\R$ be a non-zero basic linear functional. If $v\in V$ is a vector normal to $\ker(\lambda)$, then $\{v\}$ is a leaf, because it is the intersection of the unit sphere with the hyperplane $\lambda^{-1}(\lambda(v))$, both of which are union of leaves. Similarly $\{-v\}$ is a leaf, and these two point leaves are at distance $\pi$ from each other on the sphere $SV$, so that $\diam(SV/\F)\geq \pi$. On the other hand $SV/\F$ is the base of the submetry from $SV$,  which implies $\diam(SV/\F)\leq \pi$.

(c)$\Rightarrow$(d) Obvious.

(d)$\Rightarrow$(b) We use the averaging operator, see Remark \ref{R:averaging}. Let $x,y\in SV$ such that $d(L_x,L_y)>\pi/2$ and consider $f=\Av(\left<\cdot,x\right>)$. Then $f(y)<0$, so that $f$ is a non-zero basic linear functional.
\end{proof}

Most of the present article is devoted to the study of the next case, namely degree-two basic polynomials. To this end, the main structure of $\R[V]^\F$ that we exploit is that it is a \emph{transnormal algebra}, and in particular that the degree-two part is a (special, formally real) Jordan algebra. 
\begin{definition}
\label{D:transnormal}
Let $(V,\left<\ ,\right>)$ be a Euclidean vector space. A subalgebra of $\R[V]$ is called  \emph{transnormal} if it is closed under the \emph{transnormal product} 
$$ (f,g) \mapsto \left<\nabla f,\nabla g\right>$$
\end{definition}

In fact one has a transnormal product for smooth functions on an arbitrary \srf\ $(M,\F)$:
\begin{proposition}
Let $(M,\F)$ be a \srf, and $f,g\in C^\infty(M)$ be  basic functions. Then $h=\left<\nabla f,\nabla g\right>$ is a basic function.
\end{proposition}
\begin{proof}
First note that it is enough to show that $h$ is basic on the (dense) regular part. Indeed, $h$ is basic if and only if $X(h)$ vanishes identically for every smooth vertical vector field $X$.
Let $L$ be a regular leaf, and $p\in L$. Choose a simple neighbourhood $U$ of $p$, so that $\F|_U$ is given by the fibers of a Riemannian submersion $\pi:U\to\bar{U}$. Then $h|_U$ is $\F|_U$-basic, because  $h|_U=\left<\nabla\bar{f},\nabla\bar{g}\right>\circ\pi$, where $\bar{f},\bar{g}\in C^\infty(\bar{U})$ are such that $f|_U=\bar{f}\circ\pi$ and $g|_U=\bar{g}\circ\pi$. Covering $L$ with such neighbourhoods $U$, we conclude that $h$ is basic on the arbitrary regular leaf $L$, hence everywhere.  
\end{proof}
When the leaves of $(M,\F)$ have basic mean curvature vector fields, in particular for infinitesimal foliations, $C^\infty(M)^\F$ is invariant under the Laplacian operator \cite[Lemma 3.2]{LytchakRadeschi15}. This is stronger than being a transnormal algebra, because  
\[\left<\nabla f,\nabla g\right>=\frac{\Delta(fg)-\Delta(f)g-g\Delta(f)}{2}.\]

\section{Classification via Jordan algebras}
\label{S:classification}
In this section, we make the connection between Jordan algebras and quadratic basic polynomials, and prove Theorems \ref{MT:srf} and \ref{MT:classification}. See Appendix \ref{A:jordan} for the facts about Jordan algebras that will be used below.

Let $(V,\F)$ be an infinitesimal foliation, and denote by $J=\R[V]^\F_2$ the space of quadratic basic polynomials on $V$. The space of all quadratic polynomials $\R[V]_2$ is isomorphic to the space  $\Sym^2(V)$ of symmetric endomorphisms of $V$ via the map $f\mapsto \Hess(f)/2$. Composing this with the inclusion $\R[V]^\F_2\to \R[V]_2$, we define the injective linear map $\eta$ by
 \[ \eta : J \to \Sym^2(V) \qquad f\mapsto \Hess(f)/2.\]

The transnormal product (see Definition \ref{D:transnormal}) on the algebra of basic polynomials given by $(f,g)\mapsto \left<\nabla f,\nabla g\right>$ leaves invariant the subspace $J=\R[V]^\F_2$.  We claim that $J$ with this product is a Jordan algebra, and that the map $\eta$ is a homomorphism of Jordan algebras, where $\Sym^2(V)$ is endowed with the standard Jordan product, namely the symmetrization of composition $A\bullet B=(AB+BA)/2$. 

Indeed, if $f,g\in J$ and $A=\eta(f)$, $B=\eta(g)$, then 
\[  f(x)=\left< x,Ax\right>,\ g(x)=\left< x,Bx\right>. \]
In particular, $\left<\nabla f,\nabla g\right>(x)=\left< x,ABx\right>=\left<x,Cx\right>$, where $C\in\Sym^2(V)$ is given by
\[ C=\frac{AB+(AB)^T}{2}=\frac{AB+BA}{2}. \]

Before proving Theorem \ref{MT:classification}, we describe standard diagonal foliations and Clifford foliations, which are the building blocks in the classification. In particular, we describe the Jordan algebra $J$ and the embedding $\eta:J\to \Sym^2(V)$ for each of these foliations.

\begin{example}[Standard diagonal representations]
\label{E:standard}
Let $k,n\geq 1$ be integers, and $\K=\R,\C$, or $\HH$.
Consider $G=O(k)$ (respectively $U(k),Sp(k)$) with its natural action on $V=\K^k$. We call \emph{standard diagonal representations} the diagonal action of $G$ on $n$ copies of $\K^k$. It is a classical result that its algebra of invariants is generated in degree two, and that the degree-two invariants $J$ are in bijective correspondence with the set $H_n(\K)$ of $n\times n$ Hermitian matrices  with entries $A_{ij}\in\K$ via 
$$ (v_1,\ldots v_n)\mapsto \sum_{i,j} A_{ij} v_i \bar{v_j}.$$
More explicitly, if $e_1, \ldots e_k$ denotes the standard basis of $\K^k$, then in the basis
\[
(e_1, 0, \ldots 0), (e_2,0,\ldots ,0), \ldots, (e_k, 0,\ldots, 0), (0,e_1, 0, \ldots 0), (0, e_2,0,\ldots ,0), \ldots
\]
of $V$, the Hermitian matrix $A=(a_{ij})$ corresponds to the quadratic form on $V$ associated to the $kn\times kn$ matrix 
\[
A\otimes I_k=\begin{pmatrix}
a_{11} I_k & a_{12} I_k & \ldots \\
\vdots & & \ddots
\end{pmatrix}
\]
where $I_k$ denotes the $k\times k$ identity matrix. That is, the embedding $\eta:J\to \Sym^2(V)$ equals the composition of  $\otimes I_k:J=H_n(\K)\to H_{kn}(\K)$ with the inclusion $H_{kn}(\K)\subset \Sym^2(V)$.

The complex version of this result was proved by Weyl under the name of First Fundamental Theorem of $\operatorname{GL}(n,\C)$, $\OO(n,\C)$ and $\Sp(n,\C)$, see Section \ref{S:fft} for more details.
\end{example}

\begin{example}[Clifford foliations]
\label{E:Clifford}
Clifford foliations form a class of mostly inhomogeneous infinitesimal foliations, introduced in \cite{Radeschi14}. Recall that a Clifford system $C=(P_0,\ldots P_m)$ on the vector space $V=\R^{2l}$ is a set of symmetric endomorphisms of $\R^{2l}$ such that $(P_iP_j+P_jP_i)/2=\delta_{ij}I$ for every $0\leq i, j\leq m$. The partition $\F_C=\Le(|x|^2,\left<P_0x,x\right>,\ldots \left<P_mx,x\right>)$ of $\R^{2l}$ is called the \emph{Clifford foliation} associated to $C$. Here $x=(x_1,\ldots, x_{2l})$ are the standard coordinates on $\R^{2l}$. If $\F_C$ has connected leaves, then it is an infinitesimal foliation. The only cases where $\F_C$ does not have connected leaves are $C=C_{1,1},C_{1,2},C_{3,1},C_{7,1}$ (see Theorem A(1) and table on page 1665 in \cite{Radeschi14}). If $C=C_{1,1}$, then the leaves of $\F_C$ are pairs of antipodal points. If $C=C_{1,2},C_{3,1},C_{7,1}$, taking connected components of the leaves of $\F_C$ one obtains other Clifford foliations, namely the ones  associated to $C_{2,1},C_{4,1},C_{8,1}$, see \cite[Proposition 2.4]{Radeschi14}.

The Jordan algebra $J=\R[V]^\F_2$ is spanned by $|x|^2$ and $\left<P_i x,x\right>$ for $i=0,\ldots, m$. It  is isomorphic to the Spin factor $\mathcal{J}\Spin _{m+1}\subset \Cl(\R^{m+1},q)$ associated to the quadratic form $q$   given by the negative of the standard squared norm  on $\R^{m+1}$. The embedding $\eta$ is the restriction to $\mathcal{J}\Spin _{m+1}$ of the representation of the Clifford algebra $\Cl(\R^{m+1},q)$ on $V$ determined by the Clifford system $C$.

In Section \ref{S:fft} we establish the ``First Fundamental Theorem'' for Clifford foliations, namely, that the algebra of $\F_C$-basic polynomials is generated by the quadratic polynomials $|x|^2,\left<P_0x,x\right>,\ldots \left<P_mx,x\right>$.
\end{example}

\begin{proof}[Proof of Theorem \ref{MT:classification}]

Since the Jordan algebra $J=\R[V]^\F_2$ embeds in $\Sym^2(V)$, it is a special, formally real Jordan algebra.
By Theorems \ref{T:Albert} and \ref{T:JvNW}, $J$ is the direct sum of ideals $J=J_1\oplus\cdots\oplus J_r$, each of which isomorphic to $H_n(\K)$ for $\K=\R,\C,\HH$ or $\mathcal{J}\Spin_{m+1}$.

The inclusion of Jordan algebras $\eta: J\to \Sym^2(V)$ extends uniquely to a morphism of associative algebras $\mathcal{U}(J)\to\End(V)$, where $\mathcal{U}(J)$ denotes the universal enveloping algebra of $J$. By Proposition \ref{P:enveloping}(a), we have $\mathcal{U}=\mathcal{U}(J_1)\oplus\cdots\oplus\mathcal{U}(J_r)$, and therefore also $V=V_1\oplus\cdots\oplus V_r$ (see \cite[page 653]{Lang}). With respect to this decomposition of $V$, the inclusion $J=J_1\oplus\cdots\oplus J_r\to \Sym^2(V)$ is block diagonal.
The partition $\F_2=\Le(J)$ of $V$ decomposes accordingly as a product $\Le(J_1)\times\cdots\times\Le(J_r)$.

If $J_i$ is a spin factor $\mathcal{J}\Spin_{m+1}$, a basis of $\R^{m+1}\subset J_i$ is mapped to a \emph{Clifford system} (see \cite{Radeschi14}), and therefore $\Le(J_i)$ is a Clifford foliation, see Example \ref{E:Clifford}. This also covers the case $J_i=H_2(\HH)$, because $H_2(\HH)$ is isomorphic to the spin factor $\mathcal{J}\Spin_5$.

Now consider the remaining cases $J_i=H_n(\K)$ for $\K=\R,\C,\HH$, and $(n,\K)\neq (2,\HH)$. By Proposition \ref{P:enveloping}(b), $\mathcal{U}=\K^{n\times n}$. This associative algebra is simple, and its unique irreducible representation is the standard action on $\K^n$ (see \cite[page 653]{Lang}). Therefore $V_i$ is isomorphic as a real vector space to $\K^{nk}$, such that the inclusion $J_i\to\Sym^2(V_i)$ is given by taking the tensor product with the $k\times k$ identity matrix. Thus $\Le(J_i)$ coincides with the orbit decomposition of a standard diagonal representation, see Example \ref{E:standard}.
\end{proof}

\begin{proof}[Proof of Theorem \ref{MT:srf}]
By Theorem \ref{MT:classification}, we have a decomposition $V=V_1\oplus\cdots\oplus V_r$, and $\F_2$ decomposes accordingly as a product of (the orbit decompositions of) standard diagonal representations and Clifford foliations. 

Taking connected components of $\F_2$ corresponds to taking connected components in each factor. In the case of homogeneous foliations, taking connected components simply means considering the orbits of the sub-action by the identity component of the group. For the Clifford foliations, as mentioned in Example \ref{E:Clifford}, taking connected components always gives rise to a (possibly different) Clifford foliation, hence an infinitesimal foliation.

%
%
\end{proof}

When $\F$ is homogeneous, $\F_2$ is again homogeneous, and can be described using Schur's lemma. Recall that irreducible representations of compact groups $G$ on real vector spaces  $V$ are partitioned into three types: real, complex, and quaternionic. They are characterized by the fact that $V$ has type $\K$ if and only if its algebra of $G$-equivariant endomorphisms is isomorphic to $\K$, where $\K=\R$, $\C$, or $\HH$. For definitions and properties, see \cite{BrockertomDieck} section 2.6, in particular Table and Definitions (6.2) and Theorem 6.7.
\begin{proposition}[Schur's Lemma]
\label{P:homogeneous}
Let $(V,\F)$ be given by the orbits of a representation of the compact connected Lie group $G$, and  let $(V, \F_2)$ be as above. Then the indecomposable factors of $\F_2$ coincide with the isotypical components of the $G$-representation.

Moreover, for each isotypical component made up of $n$ copies of an irreducible representation $W$, the corresponding indecomposable factor of $\F_2$ is given by a standard diagonal representation on $n$ copies of $\K^k$ for some $k$, where $\K=\R,\C,\HH$ according to the type of $W$. 
\end{proposition}
\begin{proof}
Let $\phi:V\to V$ be a $G$-equivariant endomorphism. By Schur's lemma, $\phi$ is block diagonal with respect to the decomposition of $V$ into isotypical components.
Given an isotypical component of the form $W^{\oplus n}$, where $W$ is irreducible, it follows from \cite[Theorem 6.7]{BrockertomDieck} that $\phi$ is given by tensoring a matrix in $\GL(n,\K)$ with $I_k$, where $\K$ is the type of $W$. In particular, the space of symmetric $G$-equivariant endomorphisms is isomorphic to $H_n(\K)$. Since this space corresponds to $\R[V]^G_2$, the restriction of $\F_2$ to this isotypical component is defined by the standard diagonal representation on $(\K^k)^n$, for some $k$.
%
\end{proof}

As an application of Proposition \ref{P:homogeneous}, if a representation has algebra of invariants generated in degree two, then it is orbit-equivalent to a standard diagonal representation.

\section{Invariant subspaces}
\label{S:invariant}
In this section we extend a few definitions from Representation Theory to the realm of infinitesimal foliations. We start with invariant subspaces, and show that they correspond to the idempotents in the associated Jordan algebra, leading to the proof of Theorem \ref{MT:characterization} from the Introduction. We also extend the notions of isotypical components and type, and give a description of the moduli space of invariant subspaces.

\begin{definition}[Invariant subspaces]
Let $(V,\F)$ be an infinitesimal foliation. A subspace $W\subset V$ is called ($\F$-)\emph{invariant} if it is a union of leaves. The foliation $(V,\F)$ is called \emph{irreducible} if the only invariant subspaces are $\{0\}$ and $V$.
\end{definition}
These coincide with the usual definitions when $\F$ is homogeneous.
One way of generating invariant subspaces is to take the span of a set of leaves:
\begin{lemma}
\label{L:span}
Let $(V,\F)$ be an infinitesimal foliation, and let $X\subset V$ be a union of leaves (that is, a saturated set). Then $W=\Span(X)$ is an invariant subspace.
\end{lemma}
\begin{proof}
We use the averaging operator, see Remark \ref{R:averaging}. More precisely, let $W=\Span(X)$ and $f=\Av(d(\cdot, W)^2)$. Note that $f$ is a basic homogeneous degree $2$ polynomial, and that $f\geq 0$. Moreover, $f$ vanishes identically on $X$, because given $x\in X$, the leaf $L_x$ is contained in $W$, and hence $f(x)=\int_{y\in L_x}d(y,W)^2=0$. Since $f\geq 0$, the zero set of $f$ is a linear subspace, and hence $f$ vanishes on $W$. 

On the other hand, if $x\in V\setminus W$, then $f(x)=\int_{y\in L_x}d(y,W)^2>0$, because  the integrand is non-negative, and positive at $x$. Therefore $W$ equals the zero set of $f$, and is hence an invariant subspace. \end{proof}
In particular, if $W,W'$ are invariant subspaces, then $W\oplus W'=\Span(W\cup W')$ is also invariant. 

Next, we show that the invariant subspaces both determine and are determined by the degree-two basic polynomials, in the following sense:
\begin{lemma}[Idempotents]
\label{L:idempotents}
Let $(V,\F)$ be an infinitesimal foliation, and let $J=\R[V]^\F_2$ be the Jordan algebra of degree-two homogeneous basic polynomials. Then 
\begin{enumerate}[a)]
\item There is a bijective correspondence between idempotents in $J$ and $\F$-invariant subspaces, given by
\begin{equation*} f\in J \longmapsto f^{-1}(0)= \{v\in V\ |\ f(v)=0\}. \end{equation*}
\item For any $f\in\R[V]^\F_2$, the eigenspaces of $\eta(f)=\Hess(f)/2$ are invariant subspaces, and $f$ is a linear combination of idempotents.
\end{enumerate}
\end{lemma}
\begin{proof}
\begin{enumerate}[a)]
\item
Since $f$ is basic, $f^{-1}(0)$ is a union of leaves. Let $A=\eta(f)=\Hess(f)/2$, so that $f(x)=\left<x,Ax\right>$. Then $f\bullet f=f$ implies $A^2=A$, so that $A$ is orthogonal projection onto the subspace $\operatorname{im}(A)$, and $f(x)=|Ax|^2$. This shows that $f^{-1}(0)=\ker(A)$ is a vector subspace, and that the correspondence is injective.

On the other hand, given an invariant subspace $W\subset V$, the function $f(x)=d(x,W)^2$ is a basic degree-two polynomial, whose zero set equals $W$. It is idempotent because it corresponds to the orthogonal projection onto $W^\perp$.

\item Let $\lambda_1<\lambda_2<\cdots<\lambda_s$ be the eigenvalues of $A=\eta(f)$. We use induction on $s$. If $s=1$, $f$ is a multiple of the squared norm function, hence basic. Assume $s>1$. Denoting by $W_i$ the eigenspaces, it follows that $f(x)=\sum_{i=1}^s\lambda_i d(x,W_i^\perp)^2$. Then $f(x)-\lambda_1\vert x\vert^2$ is again basic, and has $W_1$ as its zero set. Thus $W_1$ is invariant. Applying the inductive hypothesis  to $f|_{W_1^\perp}$, it follows that $W_2,\ldots, W_s$ are invariant subspaces as well.
\end{enumerate}
\end{proof}
As a corollary, $W$ invariant implies $W^\perp$ invariant, so that $V$ decomposes into an orthogonal direct sum of irreducible invariant subspaces. Note however that the restrictions of $\F$ to its irreducible components do not determine $\F$.

\begin{proof}[Proof of Theorem \ref{MT:characterization}]
Denote by $\F_2^0$ the foliation given by the connected components of the leaves of $\F_2$. Let $\F'$ be an infinitesimal foliation of $V$ such that every $\F$-invariant subspace is also $\F'$-invariant. By Lemma \ref{L:idempotents}, this implies that every idempotent $f\in\R[V]^\F_2$ belongs to $\R[V]^{\F'}$, and hence  $\R[V]^\F_2\subset\R[V]^{\F'}$. Therefore $\F_2$ is coarser than $\F'$. Since the leaves of $\F'$ are connected, $\F_2^0$ is also coarser than $\F'$.

It remains to prove that $\F_2^0$ has the same invariant subspaces as $\F_2$ or, equivalently, by Lemma \ref{L:idempotents}, that their algebras of basic polynomials have the same degree-two part. Let $g$ be a quadratic $\F^0_2$-basic polynomial. Since $\F_2^0$ is coarser than $\F$, then $g\in \R[V]^{\F}_2$ which is by definition of $\F_2$ included in $\R[V]^{\F_2}_2$. The other inclusion is obvious.
\end{proof}

\begin{definition}[Isotypical components]
\label{D:isotypical}
Let $(V,\F)$ be an infinitesimal foliation, and let $\F_2$ be the associated  foliation from Theorem \ref{MT:classification}. The indecomposable factors $V_i$ of $(V,\F_2)$  are called the \emph{isotypical components} of $(V,\F)$.
\end{definition}
This coincides with the usual notion in the homogeneous case, see Proposition \ref{P:homogeneous}.

\begin{remark} Each isotypical component has a \emph{type} according to the classification. For homogeneous foliations only the \emph{real}, \emph{complex}, and \emph{quaternionic} types are possible, whereas for inhomogeneous foliations a new  type emerges, which we call \emph{Clifford} type. Note however that a foliation does not necessarily have a unique type. In fact, uniqueness is not to be expected, since there are orbit-equivalent representations of different (representation-theoretic) types, for instance the actions of $SO(4n),SU(2n),Sp(n)$ on $\R^{4n}$.

There are infinitely many inhomogeneous infinitesimal foliations of Clifford type besides Clifford foliations themselves. Concretely, if $P_0, \ldots P_m$ is any Clifford system, take \[
\F=\Le \left(|x|^2,\left<P_0x,x\right>,\ldots,\left<P_{m-2}x,x\right>, \left<P_{m-1}x,x\right>^2+\left<P_mx,x\right>^2\right).\]
This is an example of a composed foliation, see \cite{Radeschi14,GorodskiRadeschi15}. By Corollary \ref{C:composed}, the degree-two invariants correspond to the sub- Clifford system $P_0, \ldots P_{m-2}$, which is inhomogeneous when $m\geq 7$. This implies that $\F$ is inhomogeneous by Proposition \ref{P:homogeneous}. In contrast, the authors do not know of any inhomogeneous infinitesimal foliation of non-Clifford type.
\end{remark}

The Grassmannians $\Gr_j(\K^n)$ of $j$-dimensional subspaces of $\K^n$ for $\K=\R,\C,\HH$ have been classically defined as the sets of idempotents in the Jordan algebras $H_n(\K)$. The same is true of the Cayley plane, given by the primitive idempotents in the exceptional Jordan algebra $H_3(\mathbb{O})$, see \cite[page 28]{McCrimmon} and \cite[page 21]{baez}. In a similar spirit, we will use Lemma \ref{L:idempotents} to identify the moduli spaces of invariant subspaces of an infinitesimal foliation $(V,\F)$ in terms of Grassmannians and spheres. Denote by
\[\mathcal{M}_d(V,\F)=\{W\subset V\ |\ W \text{ invariant subspace and }\dim W=d\}\subset\Gr_d(V)\]
 the moduli space of $d$-dimensional invariant subspaces, and let 
 \[\mathcal{M}(V,\F)=\bigsqcup_{d=0}^{\dim V}\mathcal{M}_d(V,\F).\]
\begin{proposition}[Moduli spaces of invariant subspaces]
\label{P:moduli}
Let $(V,\F)$ be an infinitesimal foliation with isotypical components $V=V_1\oplus\cdots\oplus V_r$.
\begin{enumerate}[a)]
\item The invariant subspaces $W\subset V$ are precisely the direct sums $W=W_1\oplus\cdots\oplus W_r$ of invariant subspaces $W_i\subset V_i$. That is, $\mathcal{M}(V,\F)=\prod_i\mathcal{M}(V_i,\F|{V_i})$
\item Suppose $(V_i,\F_2|_{V_i})$ is given by the orbit decomposition of the standard diagonal representation on $V_i=(\K^k)^n$, where $\K=\R$, $\C$, or $\HH$. Then $\mathcal{M}_d(V_i,\F|_{V_i})$ is empty if $k$ does not divide $d$, and $\mathcal{M}_{jk}(V_i,\F|_{V_i})$ is a smooth submanifold of $\Gr_{jk}(V)$ diffeomorphic to $\Gr_j(\K^n)$, for $j=0,\ldots n$.

\item Suppose $(V_i,\F_2|_{V_i})$ is the Clifford foliation defined by a Clifford system $P_0,\ldots, P_m$ in $\R^{2l}$. Then $
\mathcal{M}(V_i,\F|_{V_i})$ equals
\[
\mathcal{M}_0(V_i,\F|_{V_i})\sqcup \mathcal{M}_l(V_i,\F|_{V_i})\sqcup \mathcal{M}_{2l}(V_i,\F|_{V_i})=
\{0\}\sqcup S^m \sqcup \{V_i\}.
\]
\end{enumerate}
\end{proposition}
\begin{proof}
\begin{enumerate}[a)]
\item
The decomposition of $V$ into isotypical components corresponds to the decomposition of the Jordan algebra $J=\R[V]^\F_2$ into a direct sum of ideals $J=J_1\oplus\cdots\oplus J_r$, see the proof of Theorem \ref{MT:classification}. Thus the idempotents $f\in J$ are exactly the sums $f_1+\cdots + f_r$, where each $f_i\in J_i$ is idempotent. By Lemma \ref{L:idempotents}, this means the invariant subspaces of $V$ are of the stated form.
\item By Lemma \ref{L:idempotents}, the moduli space $\mathcal{M}(V_i,\F|_{V_i})$ is diffeomorphic to the set of  idempotent matrices in $H_n(\K)$, that is, orthogonal projections onto subspaces of $\K^n$. Such orthogonal projections are in one-to-one correspondence with their images, that is, with the set of all subspaces of $\K^n$. Since the map $\eta:J_i=H_n(\K)\to\Sym^2(V_i)$ is given by tensoring with the identity $I_k$, projection onto a $j$-dimensional subspace in $\K^n$ corresponds to projection onto a $jk$-dimensional subspace in $V_i=\K^{kn}$. Therefore $\mathcal{M}(V_i,\F|_{V_i})$ is the disjoint union of $\mathcal{M}_{jk}(V_i,\F|_{V_i})=\Gr_j(\K^n)$ for $j=0,\ldots, n$.
\item 
The nontrivial idempotents are exactly the elements of the form $(1+v)/2$ where $v\in \R^{m+1}$ and $q(v,v)=-1$, that is, where $v\in S^m\subset \R^{m+1}$.  For any such $v$, $P=\eta(v)\in\Sym^2(V_i)$ has eigenvalues $\pm 1$, because $P^2=I$. The eigenspaces are switched by any $Q=\eta(w)$ for $w\in S^m$ perpendicular to $v$, because $PQ=-QP$. Therefore proper (that is, $\neq0,V_i$) invariant subspaces must have dimension $l$, and $\mathcal{M}_{l}(V_i,\F|_{V_i})$ is diffeomorphic to $S^m$.
\end{enumerate}
\end{proof}
\begin{remark}
Proposition \ref{P:moduli}, together with Schur's lemma (Proposition \ref{P:homogeneous}), gives 
an alternative proof of Theorem A(2) in \cite{Radeschi14}, which does not use the inhomogeneity of the FKM isoparametric foliations \cite{FerusKarcherMunzner81}. Theorem A(2) states that Clifford foliations are inhomogeneous, with the exception of the cases $m=1,2$, and the case $m=4$ and $P_0 P_1 P_2 P_3 P_4=\pm I$. Indeed, if $\F$ is a Clifford foliation and $m\neq 1,2,4$, then the moduli space of proper invariant subspaces is diffeomorphic to $S^m$, and thus not diffeomorphic to any projective space. By Proposition \ref{P:homogeneous}, $\F$ is inhomogeneous. If $m=4$, then the Jordan algebra $J=\R[V]^\F_2$ is isomorphic to that of the quaternionic Hopf action of $\Sp(k)$ on $(\HH^k)^2$. But if $P_0P_1P_2P_3P_4\neq \pm I$, the embedding $\eta:J\to \Sym^2(V)$ is not equivalent to that of the quaternionic Hopf action. This follows from the fact that the corresponding representations of the Clifford algebra are inequivalent, see page 1666 and Proposition 5.2 in \cite{Radeschi14}. Therefore, by Proposition \ref{P:homogeneous}, $\F$ is inhomogeneous.
\end{remark}

\section{Symmetries}
\label{S:symmetries}
In this section we use degree-two basic polynomials to generate some symmetries of an infinitesimal foliation, enough to act transitively on the connected components of the moduli space of invariant subspaces. For  a description of all foliated maps between possibly different infintesimal foliations, see Appendix \ref{A:symmetry}.

We exploit the ``action'' of the Jordan algebra $J$ on the higher graded parts of $\R[V]^\F$ given by the transnormal product $(f,g)\mapsto \left<\nabla f,\nabla g\right>$ for $(\deg(f),\deg(g))=(2,d)$. It implies that, when $(V,\F)$ is not irreducible, $\R[V]^\F$ has a finer graded structure and nontrivial symmetries:
\begin{proposition}[Multi-grading]
\label{P:multi-grading}
Let $(V,\F)$ be an infinitesimal foliation, and $W\subset V$ an invariant subspace. Then 
\begin{enumerate}[a)]
\item $\R[V]^\F$ is bi-graded with respect to the decomposition $V=W\oplus W^\perp$.
\item The linear transformation $r^W_\lambda=r_\lambda:V\to V$ defined by $w_1+w_2\in W\oplus W^\perp \mapsto w_1+\lambda w_2$ takes leaves to leaves, for any $\lambda\in\R$. 
\item Orthogonal projection onto $W$ takes leaves to leaves.
\item The algebra of $\F|_W$-basic polynomials on $W$ is naturally isomorphic to the part of $\R[V]^\F$ of bi-degree $(*,0)$. 
\item Let $f\in\R[V]^\F_2$. Then the symmetric endomorphism $\Hess(f)$ takes leaves of $(V,\F)$ onto leaves.
\end{enumerate}
\end{proposition} 
\begin{proof}
Let $f\in J$ be the unique idempotent such that $f^{-1}(0)=W^\perp$. Let $x_1,\ldots x_a$ be an orthonormal basis for $W^*$, and $y_1,\ldots y_b$ for $(W^\perp)^*$. Then $f=x_1^2+\cdots + x_a^2$. 
\begin{enumerate}[a)]
\item
Let $g=g(x_1,\ldots, x_a,y_1,\ldots, y_b)$ be a basic polynomial, and let $g=g_0+\cdots + g_d$ be the unique decomposition with $j=\deg_{\{x_i\}}(g_j)$.
The transnormal product of $f$ and $g$ equals
$$\left<\nabla f,\nabla g\right>=2\sum_{i=1}^ax_i\frac{\partial g}{\partial x_i}=   2\sum_{j=1}^d j g_j$$
and is a basic polynomial. Applying this action of $f$ repeatedly to $g$ and using the Vandermonde determinant formula shows that $g_j$ is basic for each $j$.

\item By (a), the map $r_\lambda$ satisfies $r_\lambda^*(\R[V]^\F)\subset \R[V]^\F$, and thus $r_\lambda$ takes leaves into leaves (see Proposition \ref{P:foliated}(a)). If $\lambda\neq 0$, then the same applies to the inverse $r_{\lambda^{-1}}$, so that $r_\lambda$ takes leaves (on)to leaves.

In the case $\lambda=0$, let $L$ be a leaf,  $x\in r_0(L)$, and $L_x$ the leaf through $x$. We claim that $r_0(L)=L_x$. Indeed, for any $\lambda>0$ we have
\[d_H(r_0(L),L_x)\leq d_H(r_0(L), r_\lambda(L))+d_H(r_\lambda(L), L_x)\]
where $d_H$ denotes the Hausdorff distance. Since $L$ is compact, the restrictions $r_\lambda|_L$ converge uniformly to $r_0|_L$ when $\lambda\to 0$. Thus the images $r_\lambda(L)$ converge to $r_0(L)$ in the Hausdorff metric, that is, the first term above $d_H(r_0(L), r_\lambda(L))$ goes to zero when $\lambda\to 0$. On the other hand, since $L_x$ and $r_\lambda(L)$ are equidistant (because they are leaves), and $x$ is a point in $r_0(L)$, we have
\[d_H(r_\lambda(L), L_x)=d(r_\lambda(L), L_x)=d(r_\lambda(L), x)\leq d_H(r_0(L), r_\lambda(L)). \]
Thus $d_H(r_0(L), L_x)=0$, so that $L_x=r_0(L)$. 
\item This is the special case of (b) where $\lambda=0$.
\item By  (a), the  part of $\R[V]^\F$ of bi-degree $(*,0)$ is basic.  On the other hand, given an $\F_W$-basic polynomial on $W$, its composition with orthogonal projection onto $W$ is  $\F$-basic, by (c).
\item By Lemma \ref{L:idempotents}(b), the eigenspaces of $\Hess(f)$ are invariant subspaces. Then $\Hess(f)$ takes leaves to leaves because it is a composition of maps of type $r_\lambda$ from part (b).
\end{enumerate}
\end{proof}
By induction $\R[V]^\F$ is multi-graded with respect to any decomposition of $V$ into an orthogonal direct sum of invariant subspaces.

\begin{remark}
Part (b) of Proposition \ref{P:multi-grading} can be viewed as a ``Homothetic Transformation Lemma'' where a leaf is replaced with the union $W$ of leaves.
In fact, such a generalized Homothetic Transformation Lemma is valid for general, that is, non-infinitesimal, \srf s. More precisely, if $M$ is a Riemannian manifold with a \srf\  $\F$ and an immersed submanifold $N$ which is a union of leaves, then any $\lambda$-homothety in the normal direction to $N$ around a small open subset $P\subset N$ sends leaves to leaves. This can be proved either by using the Proposition above together with the Slice Theorem (see \cite{MendesRadeschi15}), or, more directly, by noting that the original argument in \cite[Lemma 6.2]{Molino} carries over to this case.
\end{remark}

\begin{remark}
\label{R:octonionicHopf}
In the homogeneous case, the endomorphisms $\Hess(f)$ in part (e) of the proposition above are exactly the equivariant symmetric endomorphisms. In particular their \emph{enveloping algebra} consists of equivariant maps, and hence of maps sending leaves to leaves. In the inhomogeneous case this is no longer always true. In fact,  take the octonionic Hopf foliation of $V=\R^{16}$, which is also given as the Clifford foliation associated to the Clifford system of type $C_{8,1}$. It is well-known to be inhomogeneous, see Examples and Remarks 4.1.1.(ii) in \cite{GromollWalschap}. The enveloping algebra of the endomorphisms of the form $\Hess(f)$ for $f\in\R[V]^\F_2$ is the full Clifford algebra, namely the set of all endomorphisms of $\R^{16}$ (see Table 1 on page 28 of \cite{LawsonMichelsohn}), which clearly contains maps that do not send leaves to leaves.
\end{remark}

Notwithstanding the remark above, one may always generate a \emph{group} of endomorphisms taking leaves to leaves from the invertible (respectively, orthogonal) endomorphisms of the form $\Hess(f)$, for $f\in\R[V]^\F_2$.
\begin{proposition}[Symmetry]
\label{P:symmetry}
Let $(V,\F)$ be an infinitesimal foliation, and let $G\subset\OO(V)$ be the closure of the group generated by the orthogonal endomorphisms of the form $\Hess(f)$, where $f\in\R[V]^\F_2$. Then
\begin{enumerate}[a)]
\item $G$ acts by foliated isometries, that is, by maps that send leaves to leaves.
\item If $V=V_1\oplus\cdots\oplus V_r$ is the decomposition of $V$ into isotypical components, then $G$ decomposes as a product $G=G\times\cdots\times G_r$ with $G_i\subset\OO(V_i)$.
\item Suppose $(V_i,\F_2|_{V_i})$ is given by the standard diagonal representation on $V_i=(\K^k)^n $, where $\K=\R,\C,\HH$. Then $G_i=\OO(n),\SU^{\pm}(n),\Sp(n)$, acting on $V_i$ by the map $A\mapsto  A\otimes I_k$. Here $\SU^{\pm}(n)$ denotes the group of unitary matrices with determinant $\pm 1$.
\item Suppose $(V_i,\F_2|_{V_i})$ is a Clifford foliation with Clifford system $P_0\ldots, P_m$, then $G_i= \operatorname{Pin}(m+1)$, acting on $V_i$ by the spin representation associated to the Clifford system.
\end{enumerate}
\end{proposition}
\begin{proof}
\begin{enumerate}[a)]
\item By Proposition \ref{P:multi-grading}(e), the endomorphisms $\Hess(f)$ send leaves to leaves, and therefore so does the group generated by them. Moreover, the limit of a sequence of linear maps sending leaves also sends leaves to leaves, see the proof of Proposition  \ref{P:multi-grading}(b), or Appendix \ref{A:symmetry}. Therefore $G$ acts by foliated isometries.
\item The endomorphisms of the form $\Hess(f)$, where $f\in\R[V]^\F_2$, are block diagonal with respect to the decomposition of $V$. Therefore so is the group $G$. 
\item Assume $\K=\R$. By the Cartan-Dieudonn\'e theorem (see page 48, Theorem 6.6 in \cite{grove}), $\OO(n)$ is generated by reflections in $\R^n$, in particular by $H_n(\R)$. Thus $G_i=\OO(n)$. If $\K=\C$, then the subgroup $\SU^{\pm}(n)$ generated by $H_n(\C)\cap \SU^{\pm}(n)$  is dense. Indeed, it is normal and infinite, and $\SU^{\pm}(n)$  is simple. Therefore $G_i=\SU^{\pm}(n)$. Similarly, if $\K=\HH$, $G_i=\Sp(n)$. In all cases, since $\eta:J=H_n(\K)\to H_{nk}(\K)\subset \Sym^2(V_i)$ is given by taking the tensor product with the identity matrix $I_k$, the action of $G_i$ is also given by tensoring with $I_k$. 
\item  The elements of $J=\mathcal{J}\Spin_{m+1}$ with orthogonal Hessians are precisely the $v\in \R^{m+1}\subset J$ with $q(v,v)=-1$, and they generate the group $\operatorname{Pin}(m+1)\subset \Cl(\R^{m+1},q)$ in the Clifford algebra, see \cite[page 14]{LawsonMichelsohn}. The representation of $\operatorname{Pin}(m+1)$ on $V_i$ is the restriction to $\operatorname{Pin}(m+1)$ of the Clifford algebra representation defined by the Clifford system, that is, by the embedding of Jordan algebras $\eta:J\to \Sym^2(V)$. This is a real spinor representation, see \cite[page 35]{LawsonMichelsohn}. 
\end{enumerate}
\end{proof}
The action of $\Pin(m+1)$ from part (d) of the above proposition was already identified as symmetries of the Clifford foliation in \cite[Section 2.1]{Radeschi14}.

\begin{proof}[Proof of Theorem \ref{MT:symmetry}]

Let $V=V_1\oplus\cdots\oplus V_r$ be the decomposition of $V$ into isotypical components. By Propositions \ref{P:symmetry} and \ref{P:moduli}, the action of $G$ on $\mathcal{M}(V,\F)$ is given by the product of the actions of $G_i$ on $\mathcal{M}(V_i,\F|_{V_i})$. Therefore, it is enough to show that $G_i$ is transitive on each connected component of $\mathcal{M}(V_i,\F|_{V_i})$, for every $i$. So we may assume that $(V,\F)$ has only one isotypical component.

Since $G$ acts on $V$ by foliated isometries, it also acts on the Jordan algebra $J=\R[V]^\F_2$. This action preserves idempotents, which, by Lemma \ref{L:idempotents}, correspond to invariant subspaces in $V$.

If $J=H_n(\K)$ for $\K=\R,\C,\HH$, then, by Proposition \ref{P:symmetry},  $G=\OO(n)$, $\SU^\pm(n)$, $\Sp(n)$, and it  acts on $J$ by conjugation. In particular, if $g\in G$ and $A\in H_n(\K)$ is idempotent, that is, the orthogonal projection onto a subspace $U\subset \K^n$, then $g$ takes $A$ to the orthogonal projection onto $g(U)$. Thus the action of $G$ on the connected components of $\mathcal{M}(V,\F)$ is given by the standard actions of $G$ on the Grasmannians $\Gr_j(\K^n)$, which are well-known to be transitive.

If $J=\mathcal{J}\Spin_{m+1}=\operatorname{span}(1)\oplus \R^{m+1}$, then,  by Proposition \ref{P:symmetry}, $G=\Pin(m+1)$. Its action on $J$ is given the so-called adjoint representation, and in particular, $v\in S^m\subset J$ acts on $\R^{m+1}$ by the \emph{negative} of the reflection in $v^\perp $, see  \cite[Proposition 2.2]{LawsonMichelsohn}. Thus $G$ acts on $\R^{m+1}$ with image $\OO(m+1)$ if $m$ is odd, and with image $\SO(m+1)$ if $m$ is even. In either case, the action of $G$ on the set $S^m\subset\R^{m+1}$ of proper idempotents is transitive.
\end{proof}

As a corollary of Theorem \ref{MT:symmetry}, if two invariant subspaces of an isotypical component have the same dimension, then the restricted foliations are isomorphic. Note that two invariant subspaces in \emph{different} isotypical components may also have isomorphic foliations. For instance, for any  orthogonal representation $V$ of a connected compact Lie group $G$, take the product representation $V\times V$ of $G\times G$. The invariant subspaces $V\times 0$ and $0\times V$ are in distinct isotypical components, yet are orbit-equivalent, that is, have isomorphic foliations by $G\times G$-orbits.

\section{Invariant subspaces in the leaf space}
\label{S:leafspace}
The key to proving Theorem \ref{MT:smoothness} is that the metric space structure of the leaf space ``detects'' invariant subspaces, and hence the degree-two basic polynomials. To make this precise, we use the following metric notion. If $Z$ is a metric space with diameter $d$, and $X\subset Z$, define $X^*=\{z\in Z\ |\ d(z,X)=d\}$. The following lemma is well-known in the case of group representations (see \cite[page 76]{GorodskiLytchak14}):
\begin{lemma} Let $(V,\F)$ be an infinitesimal foliation, and let $SV$ be the unit sphere in $V$ centered at the origin.
\label{L:pi/2}
\begin{enumerate}[a)]
\item Assume $(V,\F)$ is free of trivial factors. Then $\diam(SV/\F)=\pi/2$ if and only if there are invariant subspaces $W$ that are proper, that is, $W\neq 0,V$.
\item Assume $\diam(SV/\F)=\pi/2$. Then $X\subset SV/\F$ is the image under the natural projection of a proper invariant subspace if and only $X=X^{**}$.
\end{enumerate}
\end{lemma}
\begin{proof}
\begin{enumerate}[a)]
\item Assume $(V,\F)$ admits a proper invariant subspace $W$. Since $W^\perp$ is again invariant and proper, we have $d(L_x,L_y)=\pi/2$ for any $x\in SV\cap W$ and $y\in SV\cap W^\perp$, and thus $\diam(SV/\F)\geq \pi/2$. Since $\F$ has no trivial factors, we must have $\diam(SV/\F)\leq \pi/2$ by  Proposition \ref{P:trivial}, and therefore  $\diam(SV/\F)=\pi/2$.

For the converse, let $x,y\in SV$ such that $d(L_x,L_y)=\pi/2$. Consider the average $f=\Av(\left<\cdot,x\right>)$, see Remark \ref{R:averaging}. Then $f=0$ by Proposition \ref{P:trivial}. In particular $0=f(y)=\int_{z\in L_y}\left<z,x\right>$, and since the integrand is everywhere non-positive, it must vanish identically. This means that $L_y\subset x^\perp$, and therefore $\Span(L_y)$ is a proper invariant subspace by Lemma $\ref{L:span}$.
\item Given a proper invariant subspace $W$, it is clear that $\pi(W)^*=\pi(W^\perp)$, and therefore that $\pi(W)^{**}=\pi(W)$.

Conversely, let $X\subset SV/\F$ satisfy $X^{**}=X$, and define $W=\Span(\pi^{-1}(X))$. Note that, for any $x\in\pi^{-1}(X)$ and $y\in\pi^{-1}(X^*)$, we have $\left<x,y\right>=0$, that is, $d(x,y)=\pi/2$ . Indeed, since $d(L_x,L_y)=\pi/2$, the linear functional $f(z)=\left<z,y\right>$ is non-positive on $L_x$. But due to the lack of trivial factors, the average of $f$ must be zero,  and thus $f$ vanishes identically on $L_x$.

Using bi-linearity of the inner product, it follows that $\left<x,y\right>=0$ for all $x\in W,y\in\Span(\pi^{-1}(X^*)) $. In particular, for any $x\in W$, its image $\pi(x)$ is at distance $\pi/2$ from $X^*$, hence $\pi(x)\in X^{**}$, which by assumption equals $X$. Therefore $\pi(W)=X$, as wanted.
\end{enumerate}
\end{proof}

\begin{proof}[Proof of Theorem \ref{MT:smoothness}.]
If the leaf spaces $SV/\F$ and $SV'/\F'$ have diameter less than $\pi/2$, then $\R[V]^\F_2$ and $\R[V']^{\F'}_2$ are spanned by $|x|^2$, and $\phi^*$ induces an isomorphism because $\phi(0)=0$. 

If the diameter of $SV/\F$ and $SV'/\F'$ are equal to $\pi/2$, then, by Lemma \ref{L:pi/2}, $\phi$ takes images of invariant subspaces to images of invariant subspaces. In particular $\phi^*$ takes the distance square functions from images of invariant subspaces to functions of the same type. These correspond to the idempotent elements in $\R[V]^\F_2$ and $\R[V']^{\F'}_2$ by Lemma \ref{L:idempotents}(a), which span all of $\R[V]^\F_2$ and $\R[V']^{\F'}_2$ by Lemma \ref{L:idempotents}(b). Thus $\phi^*$ is a linear isomorphism  $\R[V']^{\F'}_2\to \R[V]^\F_2$. Finally, on the regular part the transnormal product of basic functions clearly commutes with the projections $\pi$  and with $\phi^*$. Thus, by continuity, $\phi$ is an ismorphism of Jordan algebras.
\end{proof}

\section{First Fundamental Theorems}
\label{S:fft}
 We first show that the invariants listed in Example \ref{E:standard} generate the algebra of invariants of the standard diagonal representations. This is a simple application of the analogous result in the complex setting, called the First Fundamental Theorems of $\GL(n,\C)$, $\OO(n,\C)$, and $\Sp(n,\C)$, see \cite{Weyl} and \cite{FultonHarris} Appendix F.
 
\begin{proposition}
\label{P:ffthomogeneous}
Let $k,n\geq 1$ be integers, and $\K=\R,\C$, or $\HH$.
Consider the standard diagonal representation given by the natural action of $G=\OO(k)$ (respectively $\U(k),\Sp(k)$) on $V=(\K^k)^n$. The algebra of invariant polynomials is generated by   
$$ (v_1,\ldots v_n)\mapsto \sum_{i,j} A_{ij} v_i \bar{v_j}$$
where $A=(A_{ij})$ runs through the $n\times n$ Hermitian matrices  with entries in $\K$.
\end{proposition}
\begin{proof}
The complexifications of the natural actions of $G=\OO(k)$ (respectively $\U(k),\Sp(k)$) on $\R^k$ (respectively $\C^k$, $\HH^k$) are given by the natural actions of $G^\C=\OO(k,\C)$ (respectively $\GL(k,\C)$, $\Sp(2k,\C)$) on $\C^k$, respectively $\C^k\oplus (\C^k)^*$, $\C^{2k}\oplus(\C^{2k})^*$. Therefore the complexification of the action of $G$ on $V$ is one the representations covered by the First Fundamental Theorems of Weyl, see  \cite{Weyl} and \cite[Appendix F]{FultonHarris}. In particular, the algebra of complex $G^\C$-invariants $\C[V^\C]^{G^\C}$ is generated in degree two. Since $\C[V^\C]^{G^\C}$ is the complexification of $\R[V]^G$, it follows that the algebra of real invariants $\R[V]^G$ is generated in degree two.
By Schur's lemma (see Proposition \ref{P:homogeneous}), the listed invariants consist of all the degree-two invariants, and therefore they generate the algebra of invariants.
\end{proof}

Now we proceed to the inhomogeneous case.
Let $C=(P_0,\ldots,P_m)$ be a Clifford system in $V=\R^{2l}$, and let $\psi:\R^{2l}\to\R^{m+2}$ be the map given by 
\[
\psi(x)=\left(\psi_1(x),\ldots,\psi_{m+2}(x)\right)=\left(|x|^2,\left<P_0x,x\right>,\ldots,\left<P_mx,x\right>\right)
\]
The leaves of the corresponding Clifford foliation  $\F_C$ are defined as the sets of the form $\psi^{-1}(y)$, for $y\in\R^{m+2}$, see \cite{Radeschi14}. In the language of Definition \ref{D:levelsets}, $\F_C=\Le(\psi_1,\ldots,\psi_{m+2})$. 

The first step in the proof of Theorem \ref{MT:fft} for Clifford foliations is the following lemma:
\begin{lemma}
\label{L:fractions}
With the notation above, the field of fractions $\R(V)^{\F_C}$ of the algebra of basic polynomials for the Clifford foliation $\F_C$ is generated by $\psi_1,\ldots,\psi_{m+2}$. 
\end{lemma}
\begin{proof}
After a change of basis of $V=\R^{2l}$, we may assume that $P_i$ has the following block decomposition:
\[
P_0=\begin{pmatrix}
I & 0 \\
0 & -I\end{pmatrix}
\qquad 
P_1=\begin{pmatrix}
0 & I \\
I & 0\end{pmatrix}
\qquad 
P_i=\begin{pmatrix}
0 & E_i \\
-E_i & 0\end{pmatrix}
\quad (2\leq i\leq m)
\]
where $E_2,\ldots E_m$ are skew-symmetric matrices satisfying $E_i^2=-I$ for all $i$, and $E_iE_j=-E_jE_i$ for different $i,j$. In particular each $E_i$ is orthogonal.
Choose any unit vector $v_0\in\R^l$. Then $\{v_0,E_2v_0,\ldots,E_mv_0\}$ is an orthonormal set of vectors.

Recall that $m\leq l$. We divide the proof into two cases.

\textbf{Case 1: $m<l$.} Choose  a unit vector $v_1\in\R^l$ that is orthogonal to $v_0,E_2v_0, \ldots E_m v_0$. Define $W\subset V=\R^{2l}$ as the set of vectors of the form
\[ \left(\lambda v_0\ ,\  \alpha_1 v_0 + \sum_{i=2}^m\alpha_i E_i v_0+\mu v_1\right) \]
for $\lambda,\alpha_i,\mu\in\R$. Then the restrictions of the polynomials $\psi_i$ to $W$ are given by:
\begin{align*}
\psi_1|_W &= \lambda^2+\sum_{i=1}^m\alpha_i^2+\mu^2  \\
\psi_2|_W &= \lambda^2-\sum_{i=1}^m\alpha_i^2-\mu^2 \\
\psi_3|_W &= 2\lambda \alpha_1 \\
\psi_i|_W &= -2\lambda\alpha_{i-2}  \qquad (4\leq i\leq {m+2}) 
\end{align*}
Recall that the image of $\psi$ in $\R^{m+2}$ is defined by $y_1^2-y_2^2-\cdots -y_{m+2}^2\geq 0$ and $y_1\geq 0$ \cite[Theorem A(1)]{Radeschi14}. We claim that $\psi|_W$ has the same image. Indeed, given $(y_1,\ldots,y_{m+2})$ satisfying these conditions, it follows that $y_1+y_2\geq 0$. If $y_1+y_2=0$, then we must have $y_3=\cdots=y_{m+2}=0$, and we may take $\lambda=0$, $\alpha_i=0$, and $\mu=\sqrt{(y_1-y_2)/2}$. If $y_1+y_2>0$, we may take $\lambda=\sqrt{(y_1+y_2)/2}$, $\alpha_i=-y_{i+2}/(2\lambda)$, and $\mu=\sqrt{(y_1^2-y_2^2-\cdots -y_{m+2}^2)/2}$. This means that $W$ intersects all leaves of the Clifford foliation.

Now consider the action of $\Z/2\times \Z/2$ on $W$ with generators given by $(\lambda,\alpha_i,\mu)\mapsto(-\lambda,-\alpha_i,\mu)$ and 
$(\lambda,\alpha_i,\mu)\mapsto(\lambda,\alpha_i,-\mu)$. Note that the ring of invariants is generated by the polynomials $\lambda^2, \lambda\alpha_i, \alpha_i\alpha_j, \mu^2$, and that each $\psi_i|_W$ is invariant.

On the other hand, the restrictions $\psi_i|_W$ generate the field of fractions $\R(W)^{\Z/2\times \Z/2}$. Indeed, the invariants $\lambda \alpha_i$ and $\lambda^2$ are linear combinations of the restrictions, while the remaining invariants can be written as
\[ \alpha_i\alpha_j=-\frac{(\psi_{i+2}|_W)(\psi_{j+2}|_W)}{4\lambda^2} \qquad \qquad \mu^2=\lambda^2-\psi_2|_W-\sum_{i=1}^m \alpha_i^2 \]

Let $f\in\R[V]^{\F_C}$ be a basic polynomial. Since $\psi$ separates leaves, and has constant values along the $\Z/2\times \Z/2$-orbits in $W$, the restriction of $f$ to $W$ is invariant under  $\Z/2\times \Z/2$. Since the restrictions of $\psi_i$ generate the field of fractions of the invariants, there are non-zero polynomials $Q_1,Q_2\in\R[y_1,\ldots, y_{m+2}]$ such that $f\cdot\psi^*Q_2-\psi^*Q_1$ is zero on $W$. Since this is a basic polynomial, and $W$ meets all leaves, we conclude that $f\cdot\psi^*Q_2=\psi^*Q_1$ on $V$.

\textbf{Case 2: $m=l$.} The proof is analogous to the previous case, so we omit some details. Define $W\subset V=\R^{2l}$ as the set of vectors of the form
\[ \left(\lambda v_0\ ,\  \alpha_1 v_0 + \sum_{i=2}^m\alpha_i E_i v_0\right) \]
for $\lambda,\alpha_i\in\R$. Recall that the image of $\psi$ in $\R^{m+2}$ is defined by $y_1^2-y_2^2-\cdots -y_{m+2}^2=0$ and $y_1\geq 0$. The restriction $\psi|_W$ has the same image, so that $W$ intersects all leaves of the Clifford foliation.

Consider the action of $\Z/2$ on $W$ by the antipodal map $(\lambda,\alpha_i)\mapsto (-\lambda,-\alpha_i)$. Note that the ring of invariants is generated by the polynomials $\lambda^2, \lambda\alpha_i, \alpha_i\alpha_j$, that each $\psi_i|_W$ is invariant, and that the restrictions $\psi_i|_W$ generate the field of fractions $\R(W)^{\Z/2}$. 

Let $f\in\R[V]^{\F_C}$ be a basic polynomial. Then its restriction to $W$ is invariant under  $\Z/2$, and so there are non-zero polynomials $Q_1,Q_2\in\R[y_1,\ldots, y_{m+2}]$ such that $f\cdot\psi^*Q_2-\psi^*Q_1$ is zero on $W$. Since this is a basic polynomial, and $W$ meets all leaves, we conclude that $f\cdot\psi^*Q_2=\psi^*Q_1$ on $V$.
\end{proof}

In order to prove Theorem \ref{MT:fft} for Clifford foliations from Lemma \ref{L:fractions}, we need to use the complexification of the Clifford foliation. More precisely, we need the following fact:
\begin{lemma}
\label{L:surjective}
In the notation above, the complexification $\psi^\C$ of the map $\psi$ is surjective onto $\C^{m+2}$ if $m<l$, and onto $y_1^2=y_2^2+\ldots y_{m+2}^2$ if $m=l$.
\end{lemma}
\begin{proof}
By \cite[Section 2.1]{Radeschi14} (or Proposition \ref{P:symmetry}), the group $\Pin(m+1)$ acts on $V=\R^{2l}$ by foliated isometries. The map $\psi:V\to\R^{m+2}$ is equivariant, where $\Pin(m+1)$  acts on $\R^{m+2}=\R\oplus\R^{m+1}$ via the trivial representation on $\R$, and via the standard representation of $\OO(m+1)$ or $\SO(m+1)$ on $\R^{m+1}$, according to $m$ odd or even (see proof of Theorem \ref{MT:symmetry}). Taking complexifications, this implies that $\psi^\C$ is equivariant. Assuming $m\geq 2$ (the case $m=1$ being homogeneous, where the result is well-known), the actions of $\OO(m+1,\C)$ and $\SO(m+1,\C)$ are orbit-equivalent. Therefore it is enough to show that the image of $\psi^\C$ meets every $\OO(m+1,\C)$-orbit if $m<l$, and every orbit in $y_1^2=y_2^2+\ldots +y_{m+2}^2$ if $m=l$.

By Witt's Theorem (see \cite[Theorem 5.2]{grove}), the $\OO(m+1,\C)$-orbits of $\C^{m+2}$ come in three types:
\begin{align*}
\{(y_1, 0, \ldots, 0)\}& \qquad y_1\in\C\\
    \{(y_1, y_2, \ldots, y_{m+2})\ |\ y_2^2+\ldots+y_{m+2}^2=0\}-\{(y_1, 0, \ldots, 0)\} &\qquad y_1\in\C \\
     \{(y_1, y_2, \ldots,  y_{m+2})\ |\ y_2^2+\ldots+y_{m+2}^2=s\} &\qquad y_1\in\C, s\in\C-0
\end{align*}

Take the subspace $W\subset V$ from the proof of Lemma \ref{L:fractions}. As can be seen from the formulas for $\psi_i|_W$ given there, $\psi^\C(W^\C)$ contains the subset 
\[\{(y_1, \ldots, y_{m+2})\in\C^{m+2}\ |\ y_1+y_2\neq 0\}\sqcup \{0\}\]
 if $m<l$, and its intersection with $y_1^2=y_2^2+\ldots+y_{m+2}^2$ if $m=l$. Comparing this with the description of the $\OO(m+1,\C)$-orbits above, we see that $\psi^\C(W^\C)$ meets every orbit if $m<l$, and every orbit in the cone $y_1^2=y_2^2+\ldots +y_{m+2}^2$ if $m=l$.
\end{proof}

\begin{proof}[Proof of Theorem \ref{MT:fft}]
Let $(V,\F)$ be a product of (orbit-decompositions of) standard diagonal representations and Clifford foliations. Since $\R[V]^\F$ is the tensor product of the algebras of basic polynomials of its factors, it is enough to consider the case where $\F$ is indecomposable.

If $(V,\F)$ is given by a standard diagonal representation, $\R[V]^\F$ is generated in degree two by the First Fundamental Theorems of Weyl, see Proposition \ref{P:ffthomogeneous}. 

Assume $(V,\F)=(V,\F_C)$ is a Clifford foliation. With the notation above, we show that the algebra $\R[V]^{\F_C}$ of basic polynomials for $\F_C$ is generated by $\psi_1,\ldots,\psi_{m+2}$. 

Let $f\in \R[V]^{\F_C}$. By Lemma \ref{L:fractions}, there are  polynomials $Q_1,Q_2\in\R[y_1,\ldots,y_{m+2}]$ such that $f\cdot\psi^*Q_2=\psi^*Q_1$ on $V$. Assume that $\deg(\psi^*Q_2)$ is minimal. We will show that $\psi^*Q_2$ is constant.

Considering $f,Q_1,Q_2$ as complex ploynomials, the same equation holds on the complexification $V^\C$. Therefore $Q_1$ is zero on the intersection of the zero set of $Q_2$ with the image of $\psi^\C:V^\C\to \C^{m+2}$.

\textbf{Case 1: $m<l$.} By Lemma \ref{L:surjective}, the image of $\psi^\C$ equals $\C^{m+2}$. If $Q_2$ is not constant, Hilbert's Nullstellensatz (see \cite[Theorem 1.3A]{Hartshorne}) implies that $Q_1^p$ is divisible by $Q_2$ for some $p$, and therefore that $Q_1$ and $Q_2$ have some irreducible common factor $F$ in $\C[y_1,\ldots,y_{m+2}]$, because this is a unique factorization domain. Since $Q_1,Q_2$ have real coefficients, either $F=\bar{F}$ is real, or $\bar{F}$ is another irreducible factor of $Q_1,Q_2$, and therefore so is the real polynomial $F\bar{F}$. Thus $Q_1,Q_2$ have a real common factor, contradicting the minimality of $\deg(\psi^*Q_2)$. Thus $Q_2$, and therefore $\psi^*Q_2$, are constant.

\textbf{Case 2: $m=l$.} By Lemma \ref{L:surjective}, the image of $\psi^\C$ equals the set $g(y)=0$, where $g(y)=y_1^2-y_2^2-y_3^2-\ldots -y_{m+2}^2$. If $\psi^*Q_2$ is not constant, then $Q_2$ is not constant in the quotient ring $\C[y_1, \ldots, y_{m+2}]/(g)$, and so, by the Nullstellensatz, $Q_1^p$ is divisible by $Q_2$ in  $\C[y_1, \ldots, y_{m+2}]/(g)$, for some $p$.  We may assume $m>2$, the other cases being homogeneous. Then, using the change of variables $y_1'=y_1-y_2$ and $y_2'=y_1+y_2$, one may apply Theorem 1.1 in \cite{Nagata57} to conclude that $\C[y_1, \ldots, y_{m+2}]/(g)$ is a unique factorization domain. This implies that $Q_1$ and $Q_2$ have a common factor in this quotient ring, and since $Q_1$ and $Q_2$ are real, they in fact have a real common factor. In other words, there is $F\in\R[y_1, \ldots y_{m+2}]$ which divides both $Q_1$ and $Q_2$ modulo $g$ and is not constant modulo $g$. This implies that $\psi^*F$ is not contant and divides both $\psi^*Q_1$ and $\psi^*Q_2$, contradicting the minimality of $\deg(\psi^*Q_2)$. Thus $\psi^*Q_2$ must be constant.

\end{proof}

\begin{remark}
The proof of Theorem \ref{MT:fft} for Clifford foliations presented above is analogous to the proof of the First Fundamental Theorem for $\OO(n,\C)$ acting on $(\C^n)^n$ found in \cite[Theorem 14-1.2]{KacNotes}. The subspace $W\subset \R^{2l}$ used in the proof of Lemma \ref{L:fractions} plays the role of the subspace $B_n\subset \C^{n^2}$ of upper triangular matrices.
\end{remark}

Recall the construction of composed foliations \cite{Radeschi14}. One starts with a Clifford system $C$ as above, and an infinitesimal foliation $\F_0$ of $\R^{m+1}$. The leaves of the composed foliation $\F=\F_0\circ\F_C$ of $\R^{2l}$ are then defined as the sets of the form $\psi^{-1}(L)$, where $L$ is a leaf of the foliation $(\R^{m+2},\text{trivial}\times \F_0)$.

Now suppose one has a set of homogeneous generators $\rho_1,\ldots,\rho_N$ for the algebra of $\F_0$-basic polynomials on $\R^{m+1}$. In particular $\F=\Le(|x|^2, \psi^*(\rho_1),\ldots,\psi^*(\rho_N))$, because $y_1,\rho_1,\ldots, \rho_N$ generate $\R[y_1,\ldots,y_{m+2}]^{\text{trivial}\times \F_0}$.

\begin{corollary}
\label{C:composed}
With the notations above, the algebra $\R[V]^\F$ of basic polynomials for the composed foliation $\F$ is generated by $\psi^*(y_1)=|x|^2$ and $\psi^*(\rho_1),\ldots,\psi^*(\rho_N)$.
\end{corollary}
\begin{proof}
Let $f\in\R[V]^\F$ be an arbitrary $\F$-basic polynomial. It is in particular $\F_C$-basic, and therefore by Theorem \ref{MT:fft} there exists $Q\in \R[y_1,\ldots,y_{m+2}]$ such that $f=\psi^*(Q)$. By assumption the polynomial $Q$ is $(\text{trivial}\times \F_0)$-basic on the image of $\psi$. Therefore it coincides with its $(\text{trivial}\times \F_0)$-average $\Av(Q)$ on the image of $\psi$ (see Remark \ref{R:averaging}). Since $\Av(Q)$ is basic, there is a polynomial $P\in\R[z_1, \ldots, z_{N+1}]$ such that $\Av(Q)=(y_1,\rho_1,\ldots, \rho_N)^*(R)$. Then $f=\psi^*(Q)=\psi^*(\Av(Q))=\psi^*((y_1,\rho_1,\ldots, \rho_N)^*(R))$, as wanted.
\end{proof}

We finish this section with a result related to Lemma \ref{L:surjective}.
\begin{proposition}
Let $(V,\F)$ be an infinitesimal foliation, and $\rho_1,\ldots, \rho_N$ be generators for $\R[V]^\F$. Then the image of the complexified map $\rho^\C:V^\C\to\C^N$ equals the variety of relations between $\rho_1,\ldots, \rho_N$.
\end{proposition}
\begin{proof}
We will use the averaging operator $\Av:\R[V]\to\R[V]^\F$, see \cite{LytchakRadeschi15} and Remark \ref{R:averaging}. Its complexification $\Av^\C :\C[V^\C]\to(\R[V]^\F)^\C$ is a Reynolds operator, that is, it restricts to the identity on $(\R[V]^\F)^\C$, and $\Av^\C(fg)=f\Av^\C(g)$ for every $f\in(\R[V]^\F)^\C$, $g\in\C[V^\C]$.

Let $X$ be the variety of relations, that is, the complex variety corresponding to the finitely generated algebra $(\R[V]^\F)^\C$.  The map $\rho^\C$ corresponds to the inclusion of algebras $(\R[V]^\F)^\C\subset \C[V^\C]$. Let $x\in X$, corresponding to the maximal ideal $m_x\subset (\R[V]^\F)^\C$. If $(\rho^\C)^{-1}(x)$ is empty, this means that the ideal of $\C[V^\C]$ generated by $m_x$ contains the constant polynomial $1$. Thus there are $f_1,\ldots f_s\in m_x$ and $g_1,\ldots, g_s\in \C[V^\C]$ such that 
$1=f_1g_1+\cdots+f_s g_s$. Applying the Reynolds operator $\Av^\C$ to this equation yields $1=f_1\Av^\C(g_1)+\cdots+f_s \Av^\C(g_s)\in m_x$, a contradiction. Therefore $(\rho^\C)^{-1}(x)$ is non-empty.
\end{proof}
In the homogeneous case the propositon above is a standard fact in Invariant Theory, see Theorem 3.5.(ii) in \cite{Newstead}, or Lemma 2.3.2 in \cite{DerksenKemper}. In fact, the proof presented above is similar to the latter, the only difference being that the classical averaging (Reynolds) operator is replaced with the averaging operator from \cite{LytchakRadeschi15}.

\appendix

\section{Jordan algebras}
\label{A:jordan}
In this appendix we recall some definitions and facts about Jordan algebras and Clifford algebras. See \cite{McCrimmon,LawsonMichelsohn,baez} for more information.

A \emph{Jordan algebra} (over the reals) is a real vector space $J$ together with a commutative bilinear operation $(a,b)\mapsto a\bullet b$ satisfying the Jordan identity 
\[a\bullet(b\bullet a^2)=(a\bullet b)\bullet a^2,\]
where $a^2=a\bullet a$. 
If $\mathcal{A}$ is an associative algebra with operation $(a,b)\mapsto ab$, then $\mathcal{A}^+$ denotes the Jordan algebra whose underlying vector space is the same as $\mathcal{A}$, and whose Jordan product is given by $a\bullet b=(ab+ba)/2$.

A Jordan algebra is called \emph{special} if it is isomorphic to a subalgebra of $\mathcal{A}^+$ for some associative algebra $\mathcal{A}$. Otherwise it is called \emph{exceptional}.

A Jordan algebra $J$ is called \emph{formally real} if, for every finite subset $(a_i)_i\subset J$,
$$\sum_i a_i^2=0\ \ \Longrightarrow\ \  a_i=0\ \ \forall i$$

\begin{example}[Hermitian matrices]
Let $\K=\R,\C$, or $\HH$. The set $H_n(\K)$ of $n\times n$ Hermitian matrices is closed under the Jordan operation on $\mathcal{A}^+$, where $\mathcal{A}$ is the associative  algebra of all $n\times n$ matrices with entries in $\K$. Hence $H_n(\K)$ is a special, formally real Jordan algebra.

Even though the algebra $\mathbb{O}$ of octonions (also known as Cayley numbers, see \cite{baez})  is \emph{not} associative, the space of Hermitian matrices $H_n(\mathbb{O})$ is a (formally real) Jordan algebra for $n\leq 3$ (see \cite[page 60]{McCrimmon} and \cite[page 30]{baez}).
\end{example}

\begin{theorem}[\cite{Albert34}]
\label{T:Albert}
The Jordan algebra $H_3(\mathbb{O})$ is exceptional.
\end{theorem}

\begin{example}[Spin factor $\mathcal{J}\Spin$]
Let  $q$ be a non-degenerate quadratic form on $\R^n$, and consider the (associative) Clifford algebra $\Cl(\R^n,q)$, with defining identity $vv=-q(v,v)1$ (see \cite{LawsonMichelsohn}, \cite[page 11]{baez}). Then $vw+wv=-2q(v,w)1$, so that the subspace $\mathcal{J}\Spin(\R^n,q)=\Span(1)\oplus \R^n\subset \Cl(\R^n,q)$ is closed under the Jordan operation, and is hence a special Jordan algebra. It is formally real if and only if $q$ is negative-definite, in which case we call it a \emph{spin factor} and denote it by $\mathcal{J}\Spin_n$.\end{example}

Up to taking direct sums, the examples above cover all formally real Jordan algebras:
\begin{theorem}[\cite{JvNW34}]
\label{T:JvNW}
Let $J$ be a formally real Jordan algebra. Then it is the direct sum of simple ideals, each of which is isomorphic to $H_n(\R)$, $H_n(\C)$, $H_n(\HH)$, $H_3(\mathbb{O})$, or $\mathcal{J}\Spin_n$.
\end{theorem}

For a special Jordan algebra $J$, there is an associative algebra $\mathcal{U}=\mathcal{U}(J)$, called its \emph{universal enveloping algebra}, and an embedding $J\to\mathcal{U}$ satisfying the following property: For any associative algebra $\mathcal{A}$,  any morphism of Jordan algebras $J\to \mathcal{A}^+$  extends to a unique morphism of associative algebras $\mathcal{U}\to \mathcal{A}$  (see \cite[Sec. 5]{BirkhoffWhitman49}).

We collect the following well-known descriptions of certain universal enveloping algebras for easy reference:
\begin{proposition}\ \label{P:enveloping}
\begin{enumerate}[a)]
\item If the Jordan algebra $J$ has an identity, and is the sum of ideals $J=J_1\oplus J_2$, then so is $\mathcal{U}(J)=\mathcal{U}(J_1)\oplus\mathcal{U}(J_2)$.

\item Let $\K=\R,\C$, or $\HH$. The universal enveloping algebra of $H_n(\K)$ is the algebra $\mathcal{U}=\K^{n\times n}$ of all $n\times n$-matrices with entries in $\K$, unless $(n,\K)=(2,\HH)$, in which case $\mathcal{U}=\HH^{2\times 2}\oplus\HH^{2\times 2}$. 

\item The universal enveloping algebra of the spin factor $\mathcal{J}\Spin_n$ is $\mathcal{U}=\Cl(\R^n,q)$, the Clifford algebra associated to any negative-definite quadratic form $q$ on $\R^n$.
\end{enumerate}\end{proposition}
\begin{proof}
See for instance \cite[Theorem 2]{JJ49} for (a), \cite[Theorems 7--10] {BirkhoffWhitman49}for (b), and \cite[Proposition 1.1]{LawsonMichelsohn} for (c).  \end{proof}

For a description of the Clifford algebras in term of matrix algebras, see Table I of \cite{LawsonMichelsohn}. In particular, for $q$ negative definite, $\Cl(\R^n,q)$ is simple if and only if $n\equiv 1\,(\textrm{mod }4)$.

\section{Symmetries}
\label{A:symmetry}
The goal of this appendix is to exhibit the set of all ``foliated'' linear maps between two infinitesimal foliations as an algebraic variety. More precisely, we consider \emph{two} types of ``foliated'' linear maps:
\begin{definition}
Let $(V,\F)$ and $(V',\F')$ be infinitesimal foliations. We say a linear map $\phi:V\to V'$ \emph{takes leaves into leaves} if $\phi(L)$ is contained in a leaf of $\F'$, for every leaf $L$ of $\F$. 
We say $\phi$ \emph{takes leaves (on)to leaves} if $\phi(L)$ is equal to a leaf of $\F'$, for every leaf $L$ of $\F$. We denote the sets of all such linear maps by $\Hom((V,\F),(V',\F'))$ and $\Hom^*((V,\F),(V',\F'))$, respectively.
\end{definition}

In order to describe $\Hom^*((V,\F),(V',\F'))$, we will use the adjoint of the map $\phi$, which we denote $\phi^T:V'\to V$.
\begin{proposition}
\label{P:foliated}
Let $\phi\in\Hom(V,V')$. Then, in the notation above,
\begin{enumerate}[a)]
\item $\phi$ takes leaves into leaves if and only $\phi^*(\R[V']^{\F'})\subset \R[V]^\F$.
\item $\phi$ takes leaves onto leaves if and only if both $\phi$ and its transpose $\phi^T$ take leaves into leaves.\end{enumerate}
In particular, $\Hom((V,\F),(V',\F'))$ and $\Hom^*((V,\F),(V',\F'))$ are algebraic varieties in $\Hom(V,V')$.
\end{proposition}
\begin{proof}
\begin{enumerate}[a)]
\item Assume $\phi$ takes leaves into leaves, and let $f\in\R[V']^{\F'}$. Then $\phi^*(f)$ is basic, because if $x,y\in V$ are in the same leaf, then $\phi(x),\phi(y)$ are in the same leaf, and so $f(\phi(x))=f(\phi(y))$. Conversely, assume $\phi^*(\R[V']^{\F'})\subset \R[V]^\F$, and let $x,y\in V$ in the same leaf. If $\phi(x),\phi(y)$ were on different leaves, there would exist $f\in\R[V']^{\F'}$ such that $f(\phi(x))\neq f(\phi(y))$, because the algebra of basic polynomials separates the leaves (see \cite{LytchakRadeschi15} and Theorem \ref{T:algebraicity}). This would contradict the assumption that $\phi^*(f)$ is basic, therefore $\phi$ takes leaves into leaves.
\item We divide the proof into two cases. Case 1: $\phi$ is invertible. Assume $\phi$ takes leaves onto leaves. Then the same is true of its inverse. Moreover, since $\phi$ takes leaves into leaves, $\phi^*(|x|^2)=\phi^T\phi$ is a quadratic basic polynomial on $V$. By Proposition \ref{P:multi-grading}, $\phi^T\phi$ takes leaves onto leaves, and therefore so does  $\phi^T=(\phi^T\phi)\circ\phi^{-1}$. Conversely, assume $\phi$ and $\phi^T$ take leaves into leaves. Then  $\phi^T\phi$ takes leaves onto leaves by Proposition \ref{P:multi-grading}, and so $\phi^{-1}=(\phi^T\phi)^{-1}\circ\phi^T$ takes leaves into leaves. Since both $\phi$ and $\phi^{-1}$ take leaves into leaves, they actually take leaves onto leaves.

Case 2: $\phi$ is any linear map $V\to V'$. Let $W=\ker(\phi)^\perp\subset V$ and $W'=\phi(V)\subset V'$, and denote by $i_W,i_{W'},p_W,p_{W'}$ the natural inclusions and orthogonal projections. Suppose $\phi$ takes leaves onto leaves. Then $W$ and $W'$ are invariant subspaces. The restricted map $\phi|_W:W\to W'$ takes leaves onto leaves and is invertible. By Case 1, $(\phi|_W)^T=\phi^T|_{W'}$ takes leaves into leaves. Therefore $\phi^T=i_W\circ\phi^T|_{W'}\circ p_{W'}$ takes leaves into leaves as well. Conversely, if $\phi$ and $\phi^T$ take leaves into leaves, then $W$ and $W'$ are invariant subspaces, because they are the orthogonal complements of the kernels of $\phi$ and $\phi^T$. Moreover, their restrictions $\phi|_W:W\to W'$ and $(\phi|_W)^T$ also take leaves into leaves. By Case 1, $\phi|_W$ takes leaves onto leaves, and therefore so does $\phi=i_{W'}\circ\phi|_W\circ p_W$. 
\end{enumerate}
\end{proof}
If one is given sets $\rho_1,\ldots \rho_N$ and $\rho'_1,\ldots \rho'_{N'}$ of homogeneous generators for the algebras of basic polynomials $\R[V]^\F$ and $\R[V']^{\F'}$, respectively, one may write explicit polynomial equations defining  $\Hom((V,\F),(V',\F'))$  in the following way. (And analogously for  $\Hom^*((V,\F),(V',\F'))$.)

First note that $\phi^*(\R[V']^{\F'})\subset \R[V]^\F$ if and only if $\phi^*(\rho'_i)\in\R[V]^\F$ for every $i=1,\ldots N'$. Find a basis for the graded component $\R[V]^\F_d$ of degree $d=\deg(\rho'_i)$. This can be done by starting with all monomials in $\rho_j$, for $j=1,\ldots N$, that have degree $d$, then performing Gaussian elimination. Find linear equations on $\R[V]_d$ that define the subspace $\R[V]^\F_d$. These equations, when applied to $\phi^*(\rho'_i)$, become a finite collection of degree $d$ equations on $\phi$, which are satisfied if and only if $\phi^*(\rho'_i)\in\R[V]^\F$.

We illustrate the procedure outlined above in one simple example:
\begin{example}
Let $(V,\F)$ be the orbit-decomposition of the product action of $\OO(n)\times\OO(n)$ on $V=\R^n\times\R^n$. We will describe $\End(V,\F)$ and $\End^*(V,\F)$. If $x,y$ are the coordinates on $\R^n\times\R^n$, then the algebra of invariants is generated by $|x|^2$ and $|y|^2$. Let
\[ \phi=\begin{pmatrix}
A & B\\
C & D
\end{pmatrix} \]
be the general endomorphism of $V$, where $A,B,C,D$ are $n\times n$ real matrices. The pull-back quadratic forms $\phi^*(|x|^2)$ and $\phi^*(|y|^2)$ are associated to the symmetric matrices
\[
\phi^*(|x|^2)=\begin{pmatrix}
A^TA & A^TB\\
B^TA & B^TB
\end{pmatrix}
 \qquad \phi^*(|y|^2)=\begin{pmatrix}
C^TC & C^TD\\
D^TC & D^TD
\end{pmatrix}
\]
On the other hand, a symmetric $n\times n$ matrix represents an invariant quadratic form if and only it is of the form $\operatorname{diag} (c_1I,c_2I)$, where $I$ denotes the $n\times n$ identity matrix, and $c_1,c_2\in\R$. Therefore $\phi$ sends leaves into leaves if and only if $A^TB=C^TD=0$ and $A^TA$, $B^TB$, $C^TC$ and $D^TD$ are multiples of the identity. In other words, $\phi$ takes leaves into leaves if and only it has one the forms:
\[
\begin{pmatrix}
0 & B\\
0 & D
\end{pmatrix}
\quad
\begin{pmatrix}
0 & B\\
C & 0
\end{pmatrix}
\quad
\begin{pmatrix}
A & 0\\
0 & D
\end{pmatrix}
\quad
\begin{pmatrix}
A & 0\\
C & 0
\end{pmatrix}
\]
where $A,B,C,D\in\R \OO(n)$ are scalar multiples of orthogonal matrices.
Moreover, applying the same procedure to $\phi^T$, we conclude that $\phi$ takes leaves onto leaves if and only if it has one of the forms
\[
\begin{pmatrix}
0 & B\\
C & 0
\end{pmatrix}
\quad
\begin{pmatrix}
A & 0\\
0 & D
\end{pmatrix}
\]
for $A,B,C,D\in\R \OO(n)$. In particular, every linear map that takes leaves onto leaves is equivariant with respect to some automorphism of $G=\OO(n)\times \OO(n)$.
\end{example}
\bibliographystyle{alpha}
\bibliography{refsmooth,ref}
\end{document}